\newcommand{\powser}[1]{[\![#1]\!]}
\newcommand{\pdiv}{$p$-divisible }
\newcommand{\G}{\mathbb{G}}

\newcommand{\Sect}{\mathcal{O}}
\newcommand{\Q}{\mathbb{Q}}
\newcommand{\Z}{\mathbb{Z}}
\newcommand{\QZ}{\Q_p/\Z_p}  
\newcommand{\Zp}[1]{\Z/p^{#1}}
\newcommand{\R}{\mathbb{R}}

\newcommand{\al}{\alpha}
\newcommand{\Lk}{\Lambda_k}

\newcommand{\ti}{\tilde}
\newcommand{\wti}{\widetilde}
\newcommand{\bq}{\bar{q}}
\newcommand{\bA}{\bar{A}}
\newcommand{\lra}[1]{\overset{#1}{\longrightarrow}}

\newcommand{\Prod}[1]{\underset{#1}{\prod}}
\newcommand{\Coprod}[1]{\underset{#1}{\coprod}}
\newcommand{\Colim}[1]{\underset{#1}{\colim}}
\newcommand{\Lim}[1]{\underset{#1}{\lim}}
\newcommand{\E}{E_{n}}
\newcommand{\LE}{L_{K(t)}E_n}

\newcommand{\Bprime}{B_{t}^{'}}

\newcommand{\Chat}{\hat{C_t}}
\def \mmod{/\mkern-3mu /}

\title{Transchromatic Twisted Character Maps}
\author{
Nathaniel Stapleton \\   
}
\date{\today}
\documentclass[12pt]{article}
\usepackage{amssymb,amsfonts,amsthm,amsmath,verbatim}
\usepackage[all]{xy}
\theoremstyle{definition} 

\DeclareMathOperator{\im}{im}
\DeclareMathOperator{\Fix}{Fix}
\DeclareMathOperator{\Twist}{Twist}

\DeclareMathOperator{\colim}{colim}

\DeclareMathOperator{\Iso}{Iso}
\DeclareMathOperator{\Spec}{Spec}
\DeclareMathOperator{\Spf}{Spf}

\newtheorem{thm}[subsection]{Theorem}
\newtheorem{remark}[subsection]{Remark}
\newtheorem{example}[subsection]{Example}
\newtheorem{prop}[subsection]{Proposition}

\newtheorem*{mainthm}{Theorem}
\newtheorem{definition}[subsection]{Definition}
\begin{document}
\maketitle
\section{Introduction}
In chromatic homotopy theory, there is a history of trying to understand height $n$ phenomena in terms of height $n-1$ phenomena associated to the free loop space. There is an $S^1$-action on the free loop space by rotation. This action plays a key role in topological cyclic homology and the redshift conjecture (such as in \cite{highertopologicalcyclichomology}) and in Witten's work on the elliptic genus (see \cite{Witten-index} and \cite{Witten-ellipticgenera}). In generalized character theory the $S^1$-action has been traditionally ignored. In this work, we describe a generalized character theory in which this $S^1$-action is accounted for and we explain the relationship between it and the geometry of \pdiv groups.

When a space $X$ has an action by a group $G$, there are competing notions of the free loop space of $X$. In \cite{Witten-index}, Witten introduced the notion of a ``twisted loop space'' in which the loops and group action on the space $X$ have some interplay. Given $g \in G$ such that $g^h = e$, a loop twisted by $g$ is a map
\[
s_g: \R /h\Z \lra{} X,
\]
such that
\[
s_g(t+1) = gs_g(t)
\]
for $t \in \R /h\Z$.

This construction can be easily formalized using topological groupoids. The space of all twisted loops is modeled by the topological groupoid of maps
\[
\hom_{\text{top. groupoids}}(\R \mmod \Z,X \mmod G),
\]
where the notation $X \mmod G$ means the action topological groupoid. On restriction to the constant loops inside of this topological groupoid one recovers the inertia groupoid of the $G$-space $X$:
\[
\hom_{\text{top. groupoids}}(* \mmod \Z,X \mmod G).
\]
This is the action topological groupoid for the $G$-space 
\[
\Fix(X) = \Coprod{g \in G} X^g
\]
of Hopkins, Kuhn, and Ravenel. The target of the character maps of \cite{hkr} and \cite{tgcm} are the cohomology of the geometric realization of this topological groupoid (or the Borel construction):
\[
EG\times_G \Fix(X).
\]
In this case, the constant loops carry an action by $S^1$ that is non-trivial because of the interplay between the loops and the $G$-action. The purpose of this paper is to construct a character theory for Morava $E$-theory that keeps track of this $S^1$-action on the constant loops. We call the geometric realization of the resulting topological groupoid 
\[
\Twist(X).
\]
To be a bit more precise about what these spaces are, first note the equivalence
\[
EG\times_G \Fix(X) \simeq \Coprod{[g] \in G} EC(g)\times_{C(g)} X^g,
\]
where the coproduct is taken over conjugacy classes of elements in $G$ and $C(g)$ is the centralizer of $g$. However, a conjugacy class of elements in $G$ is the same thing as a conjugacy class of maps $\Z \lra{} G$. Fix a map 
\[
\al: \Z \lra{} G.
\]
Then $B\Z = S^1$ acts on $EC(g)\times_{C(g)} X^g$ by addition through $\al$. Finally, we can define
\[
\Twist(X) = \Coprod{[g] \in G} ES^1\times_{S^1} EC(g)\times_{C(g)} X^g.
\]
A generalized character theory built using these spaces is somewhat more refined than the generalized character theory of \cite{tgcm}. In fact, the transchromatic generalized character maps of \cite{tgcm} factor through these new maps. The cohomology of these spaces turns out to be closely tied to the geometry of \pdiv groups. To go into more detail, we must discuss the relationship between generalized character theory and \pdiv groups.

The generalized character theory of Hopkins, Kuhn, and Ravenel \cite{hkr} has proved to be a powerful tool in the analysis of the Morava $E$-theory of finite groups and finite $G$-spaces. A key idea of theirs is to construct a rational algebra $C_0$ over which the \pdiv group associated to Morava $\E$ trivializes to a constant scheme. Because $C_0$ is a rational algebra the generalized character map of Hopkins, Kuhn, and Ravenel can be viewed as transchromatic in nature: its source is in chromatic layer $n$ and its target is in chromatic layer $0$. In \cite{tgcm} the author generalized their construction to all of the heights between $0$ and $n$ by constructing an $\LE^0$-algebra $C_t$ with the property that there is a pullback square of \pdiv groups
\[
\xymatrix{\G_0 \oplus \QZ^{n-t} \ar[r] \ar[d] & \G \ar[r] \ar[d] & \G_{\E} \ar[d] \\
			\Spec(C_t) \ar[r] & \Spec(L_{K(t)}\E^0) \ar[r] & \Spec{\E^0}. }
\]
Over the ring $C_t$, $\G_{\E}$ splits into the sum of a formal group of height $t$ and a constant \'etale group of height $n-t$. It is vital to the topological constructions that the \'etale part of $C_t \otimes_{\E^0} \G_{\E}$ is constant. The main theorem of this paper is that a generalized character map can be defined using a ring $B_t$ over which the \pdiv group is a non-trivial extension of a height $t$ formal group by a height $n-t$ constant \'etale \pdiv group. More precisely, there is a pullback square
\[
\xymatrix{\G_0\oplus_{\Z_{p}^{n-t}} \Q_{p}^{n-t} \ar[r] \ar[d] & \G_{\E} \ar[d] \\ \Spf(B_t) \ar[r] & \Spec(\E^0).}
\]
We construct a complete version of $C_t$, called $\Chat$. There is a canonical map $B_t \lra{} \Chat$. Pulling back over this map recovers the split \pdiv group from the first diagram above.
Recall that the transchromatic generalized character map of \cite{tgcm} that starts in chromatic height $n$ and lands in chromatic height $t$ is a map of Borel equivariant cohomology theories
\[
\Phi_G:\E^*(EG\times_G X) \lra{} \Chat \otimes_{\LE^0}\LE^*(EG \times_G \Fix_{n-t}(X)).
\]

One of the key constructions in this paper is a generalization of the construction of the functor $EG \times_G\Fix_{h}(-)$ from \cite{tgcm}. Recall that, for a finite $G$-CW complex $X$,
\[
\Fix_{h}(X) = \Coprod{\al \in \hom(\Z_{p}^{h},G)}X^{\im \al}.
\]
This is a finite $G$-CW complex and the cohomology of $EG \times_G \Fix_{h}(X)$ is the codomain of the generalized character map. There is an equivalence
\[
EG\times_G \Fix_h(X) \simeq \Coprod{[\al] \in \hom(\Z_{p}^{h},G)/\sim} EC(\im \al) \times_{C(\im \al)} X^{\im \al},
\]
where $C(\im \al)$ is the centralizer of the image of $\al$ and the disjoint union is taken over conjugacy classes of maps. An analagous space is required to construct the twisted character maps. Let $\Z_{p}^{h} \lra{\al} G$ be a continuous map of groups in which $G$ is finite and discrete. A map of this form is determined by an $h$-tuple of pairwise commuting prime-power elements in $G$. It turns out that the pushout of abelian groups
\[
\xymatrix{\Z_{p}^{h} \ar[r] \ar[d]^{\al} & \Q_{p}^{h} \ar[d] \\ \im \al \ar[r] & \im \al \oplus_{\Z_{p}^{h}} \Q_{p}^{h}}
\]
can be extended on the left to $C(\im \al)$. We define
\[
T(\al) = C(\im \al)\oplus_{\Z_{p}^{h}} \Q_{p}^{h}.
\]
There is a short exact sequence
\[
0 \lra{} C(\im \al) \lra{} T(\al) \lra{} \QZ^{h} \lra{} 0.
\]
The action on the fixed point space $X^{\im \al}$ by $C(\im \al)$ can be extended to $T(\al)$ and we define 
\[
\Twist_h(X) = \Coprod{[\al] \in \hom(\Z_{p}^{h},G)/\sim} ET(\al)\times_{T(\al)} X^{\im \al}.
\]
Analogues of this construction have shown up in work of Ganter in \cite{stringypowerops} and in unpublished work of Rezk. As discussed in the first paragraphs, it is a Borel equivariant version of a construction in equivariant loop space theory in which the free loop space $LX$ is studied by understanding the $S^1$-action by rotation on the constant loops.
There is a canonical map
\[
\Twist_{h}(X) \lra{} \Twist_{h}(*) \cong (B\QZ)^h
\]
induced by the maps $T(\al) \lra{} (\QZ)^h$ and $X \lra{} *$.
We use this to construct a map of Borel equivariant cohomology theories called the twisted character map
\[
\Upsilon_G:\E^*(EG \times_G X) \lra{} B_{t}^{*}\otimes_{\LE^*(B\QZ^{n-t})} \LE^*(\Twist_{n-t}(X)).
\]
When $X = *$ and $G = \Zp{k}$ we show that this map recovers the global sections of the map on $p^k$-torsion
\[
\G_0\oplus_{\Z_{p}^{n-t}} \Q_{p}^{n-t}[p^k] \lra{} \G_{\E}[p^k].
\]
Because there is an isomorphism
\[
\G_0\oplus_{\Z_{p}^{n-t}} \Q_{p}^{n-t}[p^k] \lra{\cong} B_t \otimes_{\E^0} \G_{\E}[p^k],
\]
one might hope that this holds for more general spaces and groups. This is the main result.

\begin{mainthm}
For all finite groups $G$, the twisted character map induces an isomorphism of Borel equivariant cohomology theories
\[
B_t \otimes_{\E^0}\Upsilon_G:B_t\otimes_{\E^0}\E^*(EG \times_G X) \lra{\cong} B_{t}^{*}\otimes_{\LE^*(B\QZ^{n-t})} \LE^*(\Twist_{n-t}(X)).
\]
\end{mainthm}

We also show how to canonically recover the transchromatic generalized character maps of \cite{tgcm} using the canonical map $B_t \lra{} \Chat$. It should be noted that, in personal correspondence, Lurie has described a method for building the transchromatic twisted character maps from the transchromatic generalized character maps of \cite{tgcm}.

\subsection*{Acknowledgements} Once again it is a pleasure to thank Charles Rezk for his assistance with this project. Rezk pointed out a version of the $\Twist(-)$ construction and suggested that it might be related to non-trivial extensions of \pdiv groups. I would like to thank Jacob Lurie for several illuminating discussions. I'd also like to thank Matt Ando, Mark Behrens, David Carchedi, David Gepner, Tyler Lawson, Haynes Miller, Olga Stroilova, and the referee for their time and helpful remarks. 

\section{Notation and Conventions}
For a scheme $X$ over $\Spec R$ and a map $R \lra{} S$ we will set
\[
S \otimes X = \Spec S \times_{\Spec R} X.
\]
At times we will be working with formal schemes and we will mean the pullback in the appropriate category of formal schemes.

Given a ring $R$ complete with respect to an ideal $I$, let
\[
\Spf_I R = \Colim{k} \big(\Spec(R/I) \lra{} \Spec(R/I^2) \lra{} \ldots \big).
\]

By $\QZ^{h}$, $\Z_{p}^{h}$, and $\Q_{p}^{h}$ we will always mean $(\QZ)^{h}$, $(\Z_{p})^{h}$, and $(\Q_{p})^{h}$. We permanantly fix basis elements $b_1,\ldots,b_h$ of $\Z_{p}^{h}$ where
\[
b_i = (0, \ldots, 0 ,1, 0,\ldots 0)
\]
where the $1$ is in the $i$th place.

We will often need to refer to an indexed collection of elements such as $\{q_1,\ldots,q_h\}$. We will often refer to this collection using a bar:
\[
\bq = \{q_1,\ldots,q_h\}.
\]
So for example
\[
\Z_p\powser{\bq} = \Z_p\powser{q_1,\ldots,q_{h}}.
\]
Given maps of abelian groups $A \lra{} B$ and $A \lra{} C$ we let $B \oplus_A C$ be the pushout of $B$ and $C$ along $A$.

\section{Transchromatic Geometry}
\subsection{Non-Trivial Extensions} \label{Bt}
We begin by constructing a universal $\LE^0$-algebra $B_t$ with the property that there is a pullback square
\[
\xymatrix{\G_0\oplus_{\Z_{p}^{n-t}} \Q_{p}^{n-t} \ar[r] \ar[d] & \G_{\E} \ar[d] \\ \Spf_{I_t + (\bq)}(B_t) \ar[r] & \Spec(\E^0).}
\]

Fix a prime $p$ and let $0 \leq t < n$. Let $\E$ be a Morava $E$-theory and $\LE$ be the localization of $\E$ with respect to Morava $K$-theory of height $t$. We have the following description of the coefficients of $\E$ and $\LE$:
\begin{align*}
\E^0 &\cong W(k)\powser{u_1, \ldots , u_{n-1}} \\
L_t = \LE^0 &\cong W(k)\powser{u_1, \ldots , u_{n-1}}[u_t^{-1}]^{\wedge}_{(p, \ldots ,u_{t-1})}. 
\end{align*}
Let $I_t = (p,u_1,\ldots,u_{t-1})$. Note that both of the rings above are complete with respect to $I_t$. Let $\G_{\E}$ be the formal group associated to $\E$ viewed as a \pdiv group. Once and for all we fix a coordinate $x$ for $\G_{\E}$. In \cite{tgcm} it is shown that $\G = L_t \otimes_{\E^0} \G_{\E}$ is the middle term of a connected-\'etale sequence of \pdiv groups
\[
0 \lra{} \G_0 \lra{} \G \lra{} \G_{et} \lra{} 0,
\]
where $\G_0$ is the formal group associated to $\LE$ from the complex orientation induced by the canonical localization map $\E \lra{} \LE$. Note that $\E$ and $\LE$ are $p$-complete for $n > 0$ and $t > 0$ so that we have the following isomorphisms
\begin{align*}
\E^*(B\QZ^h) \cong \E^*(B(S^1)^h) \\
\LE^*(B\QZ^h) \cong \LE^*(B(S^1)^h).
\end{align*}
On the other hand, for $n = 0$ or $t = 0$, $\E^*(B\QZ^h) \cong \E^*$ and $(\LE)^*(B\QZ^h) \cong \LE^*$.

We begin by constructing an $L_t$-algebra $\Bprime$ with the property that over $\Bprime$ there is a canonical diagram of the form
\[
\xymatrix{\Z_p^{n-t} \ar[r] \ar[d] & \Q_p^{n-t} \ar[d] \\
			\G_0 \ar[r] & \G.}
\]

Define
\begin{align*}
\bq &= \{q_1,\ldots,q_{n-t}\}, \\ \bA &= \{A_1,\ldots,A_{n-t}\}, \\  ([p^k]\bA - \bq) &= ([p^k]A_1 - q_1, \ldots, [p^k]A_{n-t} - q_{n-t}).
\end{align*}

Consider the ring 
\[
\Bprime =
(\Colim{k} \text{ } L_t\powser{\bq} \otimes_{\E^0\powser{\bq}} \E^0\powser{\bq,\bA}/([p^k]\bA - \bq))_{I_t+(\bq)}^{\wedge}.
\]
The colimit is completed with respect to the ideal $I_t+(\bq)$.

This construction was inspired by the work of Ando and Morava in \cite{AndoMorava} Section 5. The idea behind the construction is that the ring $\Bprime$ has $n-t$ canonical points for the formal group $\G_0 = \G_{\LE}$ given by the set $\bq$. The elements of $\bA$ are invisible to $\G_0$ in the sense that there are no maps $L_t\powser{x} \lra{} \Bprime$ such that $x \mapsto A_i$ for any $i$. We formalize all of this in the next proposition using the language of formal algebraic geometry.

\begin{prop}
The ring $\Bprime$ is the universal complete $L_t$-algebra equipped with the following diagram of sheaves of $\Z_p$-modules in the fppf topology on complete $L_t$-algebras:
\[
\xymatrix{\Z_{p}^{n-t} \ar[r] \ar[d] & \Q_{p}^{n-t} \ar[d] \\
			\Bprime \otimes \G_0 \ar[r] & \Bprime\otimes \G.}
\]
\end{prop}
\begin{proof}
We will prove this in four steps. First we will construct the map $\Z_{p}^{n-t} \lra{} \Bprime \otimes \G_0$. Then we will show that it does not extend to a map from $\Q_{p}^{n-t} \lra{} \Bprime \otimes \G_0$. Then we will show that it does extend to a map $\Q_{p}^{n-t} \lra{} \Bprime \otimes \G$. We will use Tate's Lemma $0$ from \cite{Tate-p-div}. Let $R$ be a ring complete with respect to an ideal $I$ that contains $p$ and let $[p^k](x)$ be the $p^k$-series for some formal group law. Tate's Lemma 0 states that the topologies on $R\powser{x}$ generated by the powers of $I+(x)$ and the ideals $I^k+([p^k](x))$ agree.
\\ \\
1. We construct the map $\Z_{p}^{n-t} \lra{} \Bprime \otimes \G_{0}$ as a map 
\[
\Z_{p}^{n-t} \lra{} \G_0(\Bprime).
\]
We have the isomorphism
\begin{align*}
\G_0(\Bprime) &\cong \Lim{k}\text{ }\Colim{j}\hom_{\text{cont } L_t}(L_t\powser{x}/[p^j](x),\Bprime/(I_t+\bq)^k). \\
\end{align*}
As $\G_0$ is formal and by Tate's Lemma $0$ this is 
\[
\hom(\Spf_{I_t+(\bq)}(\Bprime),\Spf_{I_t+(x)}(L_t\powser{x})) \cong \hom_{\text{cont } L_t}(L_t\powser{x},\Bprime),
\]
which is the largest ideal of definition for $\Bprime$ with the topology induced by $I_t + (\bq)$. Call this ideal $J$. The elements of $\bq$ are elements of $J$ and can be used to define the required map.
\\ \\
2. We observe that the map defined in part 1 does not extend to $\Q_{p}^{n-t}$. This is not immediately clear because we don't know that the elements of $\bA$ are not in $J$. Assume a continuous map 
\[
L_t\powser{x} \lra{x \mapsto A_i} \Bprime
\]
exists. Consider the composite
\[
L_t\powser{x} \lra{} \Bprime \lra{} \Bprime/(\bq).
\]
It factors through
\[
L_t\powser{x} \lra{x \mapsto A_i} L_t\otimes_{\E^0} \E^0\powser{A_i}/[p^{k}](A_i)
\]
for some $k$. Now, since the map is assumed to be continuous, the power series $[p^k](x)$ must be in the kernel of the map. Thus we get a map 
\[
L_t\powser{x}/[p^k](x) \lra{x \mapsto A_i} L_t\otimes_{\E^0} \E^0\powser{A_i}/[p^{k}](A_i).
\]
But this map is an inverse to the canonical quotient in the other direction and cannot exist for dimension reasons (the ring on the right has higher rank as a free $L_t$-module). \\ \\
3. The composite $\Z_{p}^{n-t} \lra{} \Bprime \otimes \G$ does extend to $\Q_{p}^{n-t}$. Tate's Lemma $0$ gives us that 
\[
\Lim{k} \text{ } \Bprime/(I_t+\bq)^k \cong \Lim{k} \text{ } \Bprime/(I_{t}^{k}+[p^k](\bq))
\]
and so
\begin{align*}
\G(\Bprime) &\cong \Lim{k}\text{ }\Colim{j} \hom_{\text{cont } L_t}(L_t \otimes_{\E^0} \E^0\powser{x}/[p^j](x),\Bprime/(I_t+\bq)^k)  \\ 
&\cong \Lim{k}\text{ }\Colim{j}  \hom_{\text{cont } L_t}(L_t \otimes_{\E^0} \E^0\powser{x}/[p^j](x),\Bprime/(I_{t}^{k}+[p^k](\bq)))
\end{align*}
does detect the elements of $\bA$. As $k$ varies we get 
\[
\Q_{p}^{n-t} = \lim \big( \ldots \lra{} (\Q_p/p^2\Z_p)^{n-t} \lra{} (\Q_p/p\Z_p)^{n-t} \lra{} (\QZ)^{n-t} \big).
\]
\\
4. By universality in the statement of the claim we mean that $\Spf_{I_t + (\bq)}(\Bprime)$ represents the functor that brings a complete $L_t$-algebra $R$ to the set of commutative squares of the form
\[
\xymatrix{\Z_{p}^{n-t} \ar[r] \ar[d] & \Q_{p}^{n-t} \ar[d] \\
			R \otimes \G_0 \ar[r] & R \otimes \G.}
\]
From the above we see that a continuous $L_t$-algebra map $B_t' \lra{} R$ induces a square of this sort. Also, each square of this sort comes from such a map. Let $R$ be a connected complete $L_t$-algebra. Given such a square, the $R$ points of the square induce a map
\[
(\Z_p)^{n-t} \lra{} \G_0(R),
\]
which is precisely a map $L_t\powser{\bq} \lra{} R$. Now the other side of the square implies that this map extends to a map 
\[
B_t' \lra{} R.
\]
\end{proof}

The ring $\Bprime$ can be realized topologically. We give the case $t = n-1$. View $\E^0(B\QZ)$ as an $\E^0(B\QZ)$-algebra through the map $\QZ \lra{\times p^k} \QZ$. Using the coordinate this map is described by 
\[
\E^0\powser{x} \lra{x \mapsto [p^k](x)} \E^0\powser{x}.
\] 
Set $q = [p^k](x)$. Then the codomain can be described as
\[
\E^0\powser{x} \cong \E^0\powser{q,A_1}/([p^k]A_1 - q)
\]
and the map is the inclusion
\[
\E^0\powser{q} \lra{} \E^0\powser{q, A_1}/([p^k]A_1 - q).
\] 
The Weierstrass preparation theorem implies that the codomain is a free module of rank $p^{kn}$ over the domain. To topologically define the maps that the colimit is taken over consider the square
\[
\xymatrix{\QZ \ar[r]^{\times p^{k+1}} \ar[d]^{\times p} & \QZ \ar[d]^{=} \\ \QZ \ar[r]^{\times p^k} & \QZ.}
\]
This induces
\begin{align*}
\LE^0(B\QZ) \otimes^{p^k}_{\E^0(B\QZ)} \E^0(B\QZ) & \\ 
\lra{1\otimes p} \LE^0(B\QZ) & \otimes^{p^{k+1}}_{\E^0(B\QZ)} \E^0(B\QZ).
\end{align*} 

Now we see that, up to a matter of completion, we have the following isomorphism
\[
\Bprime \cong \Colim{k} \text{ } \LE^0(B(\QZ^{n-t})) \otimes_{\E^0(B(\QZ^{n-t}))}^{p^k} \E^0(B(\QZ^{n-t})),
\]
where the right side of the tensor product is induced by the map
\[
\QZ^{n-t} \lra{\times p^k} \QZ^{n-t}.
\]
As $L_{K(0)}\E$ is a rational cohomology theory, this implies that $B_{0}^{'} \cong C_{0}^{'}$, where $C_{0}^{'}$ is the ring defined just before Proposition 2.13 of \cite{tgcm}.

Over $\Bprime$ there is a canonical map of \pdiv groups 
\[
\G_0\oplus_{(\Z_p)^{n-t}}(\Q_p)^{n-t} \lra{} \G.
\]
We can give a description of the global sections of this map on $p^k$-torsion. We begin with an informal description of the global sections and then describe the map. Let $R$ be a complete $\Bprime$-algebra. Let $q_j \in R$ be the image of $q_j \in \Bprime$. Then
\[
\G_0\oplus_{(\Z_p)^{n-t}}(\Q_p)^{n-t}(R) = \G_0(R) \oplus_{(\Z_p)^{n-t}}(\Q_p)^{n-t}
\]
has elements of the form 
\[
(r,\frac{i_1}{p^{k_1}},\ldots,\frac{i_{n-t}}{p^{k_{n-t}}}),
\]
where $r$ is an element of $\G_0(R)$ and the quotients are elements of $\Q_p$ greater than or equal to $0$ and less than $1$. Addition of two such tuples is computed by formal addition in the first variable and adding in $\Q_p$ in the other variables. However, if in the sum $\frac{i_j}{p^{k_j}} \geq 1$ for some $1 \leq j \leq n-t$ then we subtract $1$ from it and formally subtract $q_j$ from the first term. A more thorough discussion of the arithmetic can be found in \cite{AndoMorava} Subsection 5.1.

Thus, for example, an element of the form
\[
(r,\frac{1}{p},0,\ldots,0)
\]
is $p$-torsion if and only if $[p](r) = q_1$, for then
\[
p\cdot(r,\frac{1}{p},0,\ldots,0) = ([p](r),1,0,\ldots,0) = ([p](r) -_{\G_0} q_1,0,0,\ldots,0) = (0,0,0,\ldots,0).
\]

For $i = (i_1,\ldots,i_{n-t}) \in \Lk$, let 
\[
[i](\bq) = [i_1](q_1) +_{\G_0} \ldots +_{\G_0} [i_{n-t}](q_{n-t}).
\]
From the above discussion we deduce that
\[
\Sect_{\G_0\oplus_{\Z_{p}^{n-t}}\Q_{p}^{n-t}[p^k]} \cong \Prod{i \in \Lk}\Bprime \powser{x_i}/([p^k](x_i) - [i](\bq)).
\]
Recall that the ideals $([p^k](x_i) - [i](\bq))$ and $([p^k](x_i) -_{\G_0} [i](\bq))$ are equal.

Now we can describe the global sections of the map
\[
\G_0\oplus_{\Z_{p}^{n-t}}\Q_{p}^{n-t}[p^k] \lra{} \G_{\E}[p^k].
\]
This is a map
\[
\E^0\powser{x}/[p^k](x) \lra{} \Prod{i \in \Lk}\Bprime \powser{x_i}/([p^k](x_i) -_{\G_0} [i](\bq)).
\]
Recall from Lemma 2.16 in \cite{tgcm} that the ideals $([p^k](x_i) -_{\G_0} [i](\bq))$ and $([p^k](x_i) - [i](\bq))$ are equal. We only need to explain where $x$ maps to for each $i \in \Lk$. The map above is given by 
\[
x \mapsto (x_i -_{\G_{\E}} [i](\bA))_{i \in \Lk},
\]
where 
\[
[i](\bA) = [i_1]A_1 +_{\G_{\E}} \ldots +_{\G_{\E}} [i_{n-t}](A_{n-t})
\]
for $A_i \in \Bprime$ with the property that $[p^k]A_i = q_i$.

Recall that $A_i \in \Bprime$ is an element of $\G(\Bprime)$ but not an element of $\G_0(\Bprime)$. Now we can check that the map is well defined:
\begin{align*}
[p^k](x_i -_{\G_{\E}} [i](\bA))_{i \in \Lk} &= ([p^k]x_i -_{\G_{\E}} [p^k][i](\bA))_{i \in \Lk} \\
&= ([p^k]x_i -_{\G_{\E}} [i][p^k](\bA))_{i \in \Lk} \\
&= ([i](\bq) -_{\G_{\E}} [i](\bq))_{i \in \Lk} \\
&= 0.
\end{align*} 

We leave it as an exercise to the reader to verify that this is a map of Hopf algebras.

As in \cite{tgcm} the induced map $\QZ^{n-t} \lra{} \G_{et}$ over $\Bprime$ defines a subset $R \subset \Bprime$ such that over $B_t = (R^{-1}\Bprime)^{\wedge}_{I_t + (\bq)}$ there is a canonical isomorphism
\[
\G_0\oplus_{\Z_{p}^{n-t}}\Q_{p}^{n-t} \lra{\cong} \G.
\]

\begin{prop}
The complete $L_t$-algebra $B_t$ represents the functor
\[
\Iso_{\G_0/}(\G_0\oplus_{\Z_{p}^{n-t}}\Q_{p}^{n-t},\G):\text{Complete } L_t-\text{Alg} \lra{} \text{Set} 
\]
that sends
\[
R \mapsto \Iso_{\G_0/}(R\otimes_{L_t}\G_0\oplus_{\Z_{p}^{n-t}}\Q_{p}^{n-t},R\otimes \G).
\]
\end{prop}
\begin{proof}
A map
\[
B_t \lra{} R
\]
induces via precomposition a map
\[
\Bprime \lra{} R,
\]
which represents a square
\[
\xymatrix{\Z_{p}^{n-t} \ar[r] \ar[d] & \Q_{p}^{n-t} \ar[d] \\
			\G_0 \ar[r] & \G}
\]
over $R$. This induces a map 
\[
\QZ^{n-t} \lra{} \G_{et},
\]
which is an isomorphism by the definition of $B_t$. Now the five lemma implies that the map from the pushout of the square to $\G$ is an isomorphism.
\end{proof}

\subsection{Relation to $C_t$} 
\label{relationtoC_t}
Let $C_t$ be the $L_t$-algebra defined in Section 2.9 of \cite{tgcm}. Recall that $C_t = S^{-1}C_{t}'$. Let $\Chat ' = (C_{t}')^{\wedge}_{I_t}$ be the completion of $C_{t}'$ at the ideal $I_t$. Now we define $\Chat$ to be $(S^{-1}(C_{t}')^{\wedge}_{I_t})^{\wedge}_{I_t}$. It should not be surprising that this is the correct analogue of $C_t$ in the setting of this paper. The recompletion process just forces the ring to be a complete $L_t$-algebra.

There is a canonical topologically-induced map from $B_t$ to $\Chat$. Let $\Lk = (\Zp{k})^{n-t}$. Consider the commutative square
\[
\xymatrix{\Lk \ar[r] \ar[d] & \QZ^{n-t} \ar[d]^{\times p^k} \\ e \ar[r] & \QZ^{n-t}.}
\]
The map
\[
\LE^0(B(\QZ^{n-t})) \otimes^{p^k}_{\E^0(B(\QZ^{n-t}))} \E^0(B(\QZ^{n-t})) \lra{} L_t\otimes_{\E^0} \E^0(B(\Lk))
\]
is induced in the left and bottom factors by the bottom map of the square and in the right factor by the top map of the square. This map fits together with the colimit on both sides to provide a map $B_{t}^{'} \lra{} \Chat '$.

\begin{prop}
The map above sends the set $R \subset \Bprime$ bijectively to the set $S \subset \Chat '$.
\end{prop}
\begin{proof}
Recall from \cite{tgcm} that there is a canonical map
\[
\phi: \QZ^{n-t} \lra{} \G_{et}
\]
over $\Chat '$. Also the fixed coordinate provides an isomorphism (in the notation of \cite{tgcm})
\[
\Sect_{\G_{et}[p^k]} \cong \Chat '[y]/(j_k(y)).
\]
In Section 2.9 of \cite{tgcm} we defined 
\[
\phi_y : \Lk \lra{} \Chat '
\]
by $\phi_y(i) = \pi_i \phi^* (y)$, the image of $y$ in the component of $\Sect_{\Lk}$ corresponding to $i \in \Lk$. We then defined
\[
S = \Colim{k}\text{ }S_k = \Colim{k} \text{ } \{\phi_y(i)|i \in \Lk\}.
\]
Let $R_k \subset B_t'$ be defined analagously. Because $\phi_y$ is injective the finite sets $S_k$ and $R_k$ have the same cardinality.
Pulling the canonical map
\[
\QZ^{n-t} \lra{} B_t'\otimes \G_{et}
\]
back along the map $B_t' \lra{} \Chat '$ gives the commutative square
\[
\xymatrix{ \QZ^{n-t} \ar[r] \ar[d] & \QZ^{n-t} \ar[d] \\ \Chat '\otimes \G_{et} \ar[r] & B_t' \otimes \G_{et}.}
\]
The commutativity of the global sections of the $p^k$-torsion of the square implies that there is a surjection
\[
R_k \lra{} S_k.
\]
Because the sets are finite and have the same cardinality this implies the map is a bijection.
\end{proof}
This implies that $B_t$ is nonzero.

\begin{prop}
The map $L_t/I_t \lra{} B_t/(I_t+\bq)$ is faithfully flat.
\end{prop}
\begin{proof}
Note that 
\[
B_t/(I_t+\bq) \cong C_t/I_t.
\]
Now this follow immediately from Proposition 2.18 in \cite{tgcm}.
\end{proof}

\section{Transchromatic Twisted Character Maps}
The transchromatic twisted character map is defined to be the composition of two maps. The first map is induced by a map of topological spaces and the second one is algebraic in nature.
\subsection{The Topological Map}
The transchromatic twisted character map is the composition of two maps - a topological map and an algebraic map. In this section we describe the topological map. It is induced by a map of topological spaces
\[
BT(\gamma_k) \times_{B\QZ^{n-t}} \Twist_{n-t}(X) \lra{} EG \times_G X
\]
where $X$ is a finite $G$ CW-complex. In this map, the domain is a pullback and the codomain is the Borel construction.
Our first goal is to define the functor $\Twist_{n-t}(-)$ from finite $G$-spaces to spaces. It plays the role of the functor $EG\times_G \Fix_{n-t}(-)$ of \cite{hkr} and \cite{tgcm}. We will give two equivalent definitions of the functor. The first is simpler to understand but requires some choices. The second construction is free of any choices.

Given a continuous map of groups 
\[
\al:\Z_{p}^{h} \lra{} G,
\]
we can form the pushout in abelian groups
\[
\xymatrix{\Z_{p}^{h} \ar[r] \ar[d]^{\al} & \Q_{p}^{h} \ar[d] \\ \im \al \ar[r] & \im \al \oplus_{\Z_{p}^{h}} \Q_{p}^{h}.}
\]

\begin{prop} \label{pushout}
The pushout extends on the left to the centralizer of the image of $\al$
\[
C(\im \al) = \{g \in G | ghg^{-1} = h \text{ } \forall h \in \im \al\}.
\]
\end{prop}
\begin{proof}
Let $g,h \in C(\im \al)$, $s,t \in \Q_{p}^{h}$, and $i,j \in \Z_{p}^{h}$. We will represent an element of the pushout by $[g,t]$ with the relation
\[
[g,t+i] = [g\al(i),t].
\]
We prove that multiplication is well-defined. We have
\begin{align*}
[g,t+i][h,s+j] &= [g\al(i),t][h\al(j),s]\\
&= [g\al(i)h\al(j),t+s]
\end{align*}
and
\begin{align*}
[g,t+i][h,s+j] &= [gh,t+s+i+j]\\
&= [gh\al(i+j),t+s].
\end{align*}
Now $gh\al(i+j) = g\al(i)h\al(j)$ because $\al$ is a homomorphism and because $g,h$ are in the centralizer of the image of $\al$.
\end{proof}
\begin{definition}
Let
\[
T(\al) = C(\im \al)\oplus_{\Z_{p}^{h}} \Q_{p}^{h}.
\]
\end{definition}
\begin{prop}
There is a short exact sequence
\[
0 \lra{} C(\im \al) \lra{} T(\al) \lra{} \QZ^{h} \lra{} 0.
\]
\end{prop}
\begin{proof}
This is clear.
\end{proof}
\begin{example}
Let $\gamma_k:\Z_{p}^{h} \lra{} \Lk = (\Zp{k})^{h}$ be the quotient. Then $T(\gamma_k)$ is the pushout
\[
\Lk\oplus_{\Z_{p}^{h}} \Q_{p}^{h}.
\] 
Note that this is isomorphic to $\QZ^{h}$ and that this is the middle term of a short exact sequence
\[
\xymatrix{\Z_{p}^{h} \ar[r] \ar[d]^{\gamma_k} & \Q_{p}^{h} \ar[d] \ar[r] & \QZ^{h} \ar[d]^{=} \\ \Lk \ar[r] & \Lk\oplus_{\Z_{p}^{h}} \Q_{p}^{h} \ar[r]^(.6){\times p^k} & \QZ^{h}.}
\]
\end{example}

\begin{example} \label{presentation}
Next we work a slightly more complicated example. Let
\[
\al:\Z_{p}^{h} \lra{} \Zp{k}.
\]
We will try to understand the relationship between $T(\al)$ and $\QZ^{h}$ through the quotient map $T(\al) \lra{c} \QZ^{h}$ in terms of their duals. We call it ``$c$" for cokernal. First, the dual of $\QZ^{h}$ is $\Z_{p}^{h}$ and the fact that the quotient map is surjective implies that the dual map is injective. Now let $i_1,\ldots,i_h$ be the image of the basis elements of $\Z_{p}^{h}$ in $\Zp{k}$. Consider the inclusion
\[
\Zp{k} \lra{f} \QZ
\]
given by 
\[
i \mapsto \frac{i}{p^k},
\]
where, for convenience, we consider $0$ to map to $1 \in \QZ$. 
Elements of $T(\al)$ can be put in the form 
\[
(i,(\frac{a_1}{p^{k_1}},\ldots,\frac{a_h}{p^{k_h}})),
\]
where $i \in \Zp{k}$, $0 \leq \frac{a_j}{p^{k_j}} < 1$, and $(0,(0,\ldots,1,\ldots,0)) = (i_j,(0,\ldots,0))$ when the $1$ is in the $j$th place.
Consider the composition
\[
\xymatrix{T(\al) \ar[r]^c & \QZ^h \ar[d]^{\pi_j} \\ & \QZ,} 
\]
where $\pi_j$ is the projection onto the $j$th factor. The composition $\pi_j \circ c$ gives the image of the element $b_j = (0,\ldots,0, 1,0,\ldots,0) \in \Z_{p}^{h}$ in $T(\al)^*$, where the $1$ is in the $j$th place. We construct a map
\[
T(\al) \lra{x} \QZ
\]
by sending
\[
(i,(\frac{a_1}{p^{k_1}},\ldots,\frac{a_h}{p^{k_h}})) \mapsto f(i)+\sum_{j = 1}^{h}f(i_j)\frac{a_j}{p^{k_j}}.
\]
A quick computation gives the relation
\[
p^k x = i_1(\pi_1 \circ c)+\ldots +i_h (\pi_h \circ c).
\]
Next note that the map
\[
(\pi_1 \circ c, \ldots, \pi_h \circ c, x):T(\al) \lra{} \QZ^{h+1}
\]
is a monomorphism so the dual is an epimorphism from a free $\Z_p$-module. Thus $x$ generates the part of $T(\al)^*$ not hit by $c^*$ and satisfies the relation above. Finally we conclude that $T(\al)^*$ is the $\Z_p$-module with the following presentation
\[
\{b_1,\ldots,b_h,x|p^kx = i_1b_1+\ldots+i_hb_h\}. 
\]
\end{example}

Let $G$ be a finite group and $X$ a finite $G$-space (equivalent to a finite $G$-CW complex).

\begin{prop} \label{actionextends}
The action on the fixed point space $X^{\im \al}$ by $C(\im \al)$ extends to an action by $T(\al)$.
\end{prop}
\begin{proof}
We will represent elements of $T(\al)$ as tuples of the form $[g,t]$ where $g \in C(\im \al)$ and $t \in \Q_{p}^{h}$. We define
\[
[g,t]x = gx.
\]
We see that if $j \in \Z_{p}^{h} \subseteq \Q_{p}^{h}$ then
\[
x = [1,j]x = [\al(j),0]x = \al(j)x = x
\]
as $x \in X^{\im \al}$. Thus the action is well-defined.
\end{proof}

Recall from \cite{tgcm} that the set of continuous homomorphisms $\hom(\Z_{p}^{h},G)$ is a $G$-set under conjugation.

The following definition is fundamental to our work here:
\begin{definition}
For a $G$-space $X$, let
\[
\Twist_h(X) = \Coprod{[\al] \in \hom(\Z_{p}^{h},G)/\sim} ET(\al)\times_{T(\al)} X^{\im \al},
\]
where the coproduct is over conjugacy classes.
\end{definition}

This is analagous to the equivalence
\[
EG\times_G \Fix_h(X) \simeq \Coprod{[\al] \in \hom(\Z_{p}^{h},G)/\sim} EC(\im \al) \times_{C(\im \al)} X^{\im \al}.
\]

\begin{remark} \label{T-action}
There is an alternative way to view the relationship between $\Twist_h(X)$ and $EG\times_G \Fix_h(X)$. There is a $\mathbb{T} = (S^1)^{\times h}$ action on
\[
\Coprod{[\al] \in \hom(\Z_{p}^{h},G)/\sim} EC(\im \al) \times_{C(\im \al)} X^{\im \al}.
\]
The action is induced componentwise. We begin by treating the case when $X = *$. Thus fix an $[\al] \in \hom(\Z_{p}^{h},G)/\sim$ and consider, by precomposition with $\Z^{h} \lra{} \Z_{p}^{h}$, the map
\[
\Z^{h}\times C(\im \al) \lra{} \Z_{p}^{h} \times C(\im \al) \lra{+_{\al}} C(\im \al),
\]
where
\[
+_{\al} : (s, g) \mapsto \al(s)+g.
\]
Now applying $B(-)$ to the map gives an action of $\mathbb{T}$ on $BC(\im \al)$ and the Borel construction gives the $p$-complete equivalence
\[
E\mathbb{T} \times_{\mathbb{T}} BC(\im \al) \simeq BT(\al).
\]
It is not hard to see that the above construction extends to give a $\mathbb{T}$-action on $EC(\im \al)\times_{C(\im \al)} X^{\im \al}$. Thus, up to $p$-completion, $\Twist_h(X)$ is the homotopy orbits for a $\mathbb{T}$-action on
\[
\Coprod{[\al] \in \hom(\Z_{p}^{h},G)/\sim} EC(\im \al) \times_{C(\im \al)} X^{\im \al}.
\]
\end{remark}

\begin{example}
Note that if $X = *$ and $G = e$ then
\[
\Twist_{h}(*) \cong (B\QZ)^h.
\]
\end{example}

Now we provide a more canonical form for the functor $\Twist_h(-)$. Let $\al, \beta:\Z_{p}^{h} \lra{} G$ and let
\[
C(\al, \beta) = \{g\in G| g\al g^{-1} = \beta\}.
\]

\begin{prop}
The set $C(\al,\beta)$ is a right $\Z_{p}^{h}$-set through $\al$.
\end{prop}
\begin{proof}
For $j \in \Z_{p}^{h}$ and $g \in C(\al,\beta)$ let 
\[
gj = g\al(j).
\]
This is well-defined because
\[
g\al(j)\al\al(j)^{-1}g^{-1} = g \al g^{-1} = \beta.
\]
\end{proof}

Now let $T(\al,\beta)$ be the coequalizer of
\[
\xymatrix{C(\al,\beta) \times \Z_{p}^{h} \times \Q_{p}^{h} \ar@<1ex>[r] \ar[r] & C(\al,\beta) \times \Q_{p}^{h},}
\]
where one map is induced by the action of $\Z_{p}^{h}$ on $C(\al,\beta)$ described above and the other is induced by the standard action of $\Z_{p}^{h}$ on $\Q_{p}^{h}$.
There is a map 
\[
T(\al,\beta) \times X^{\im \al} \lra{c} X^{\im \beta}
\]
given by 
\[
[g,t]x = gx \in X^{\im g \al g^{-1}} = X^{\im \beta}.
\]
The proof that it is well defined is similar to the proof of Prop \ref{actionextends}. Note that $T(\al) = T(\al,\al)$.

There is a natural composition
\[
T(\beta, \delta) \times T(\al,\beta) \lra{} T(\al,\delta)
\]
given by
\[
[h,s] \times [g,t] \mapsto [hg,s+t].
\]
Also $T(\al,\al)$ contains an identity: the identity element $e \in G$.

\begin{definition}
Let $\Twist_{h}^{gpd}(X)$ be the geometric realization of the nerve of the topological groupoid 
\[
\xymatrix{\Coprod{(\al,\beta) \in \hom(\Z_{p}^{h},G)^{\times 2}}T(\al,\beta)\times X^{\im \al} \ar[d]_{d} \ar@<1ex>[d]^{c} \\ \Coprod{\al \in \hom(\Z_{p}^{h},G)} X^{\im \al},}
\]
where the domain map $d$ is the projection and the codomain map $c$ is the map defined above. 
\end{definition}

\begin{prop}
We have the following equivalence for all finite $G$-spaces $X$:
\[
\Twist_{h}(X) \simeq \Twist_{h}^{gpd}(X).
\]
\end{prop}
\begin{proof}
We will construct a map of topological groupoids 
\[
\xymatrix{\Coprod{[\al] \in \hom(\Z_{p}^{h},G)/\sim}T(\al)\times X^{\im \al} \ar[d]_{d} \ar@<1ex>[d]^{c} \ar[r]^{f_{mor}} & \Coprod{(\al,\beta) \in \hom(\Z_{p}^{h},G)^{\times 2}}T(\al,\beta)\times X^{\im \al} \ar[d]_{d} \ar@<1ex>[d]^{c} \\ \Coprod{[\al] \in \hom(\Z_{p}^{h},G)/\sim} X^{\im \al} \ar[r]^{f_{ob}} & \Coprod{\al \in \hom(\Z_{p}^{h},G)} X^{\im \al}}
\]
that is an equivalence in the sense of Corollary 4.8 and Definition B.15 in \cite{Carchedi}. 

Fix a collection of representatives of the conjugacy classes of $\hom(\Z_{p}^{h},G)/\sim$. On objects we'll send
\[
\Coprod{[\al] \in \hom(\Z_{p}^{h},G)/\sim}X^{\im \al} \lra{} \Coprod{\alpha \in \hom(\Z_{p}^{h},G)} X^{\im \al}
\]
by using the representatives of the conjugacy classes.

The map on morphisms
\[
\Coprod{[\al] \in \hom(\Z_{p}^{h},G)/\sim} T(\al) \times X^{\im \al} \lra{} \Coprod{(\alpha,\beta)}  T(\alpha,\beta) \times X^{\im \alpha}
\]
sends $([g,t],x) \in T(\al) \times X^{\im \al} \mapsto ([g,t],x) \in T(\al,\al) \times X^{\im \al}$. The commutativity of the necessary diagrams is clear.

To ease the notation let 
\[
(H_0,H_1) = \Big(\Coprod{[\al] \in \hom(\Z_{p}^{h},G)/\sim} X^{\im \al},\Coprod{[\al] \in \hom(\Z_{p}^{h},G)/\sim}T(\al)\times X^{\im \al}\Big)
\]
and
\[
(G_0,G_1) = \Big(\Coprod{\al \in \hom(\Z_{p}^{h},G)} X^{\im \al},\Coprod{(\al,\beta) \in \hom(\Z_{p}^{h},G)^{\times 2}}T(\al,\beta)\times X^{\im \al}\Big).
\]
The map of topological groupoids above is then just
\[
\xymatrix{H_1 \ar[d]_{d} \ar@<1ex>[d]^{c} \ar[r]^{f_{mor}} & G_1 \ar[d]_{d} \ar@<1ex>[d]^{c} \\ H_0 \ar[r]^{f_{ob}} & G_0.}
\]
First consider the pullback 
\[
G_1\times_{G_0}H_0
\]
along the maps $d$ and $f_{ob}$. To show that $f$ is essentially surjective we must show that 
\[
G_1\times_{G_0}H_0 \lra{c \circ \pi_1} G_0
\]
admits local sections, where $\pi_1$ is the projection onto the first factor. Since $f_{ob}$ maps components homeomorphically to components the pullback is just the components of 
\[
\Coprod{(\al,\beta) \in \hom(\Z_{p}^{h},G)^{\times 2}}T(\al,\beta)\times X^{\im \al}
\]
with $\al$ in the chosen reprentatives of the conjugacy classes. Now for a fixed $\beta:\Z_{p}^{h} \lra{} G$, choosing an element of $[g,t] \in T(\al,\beta)$ allows one to construct a section of $c$,
\[
X^{\im \beta} \lra{} T(\al,\beta) \times X^{\im \al},
\]
mapping
\[
x \mapsto ([g,t], g^{-1}x).
\]

To show that $f$ is fully faithful we must show that the following square is a pullback:
\[
\xymatrix{H_1 \ar[d]^{(d,c)} \ar[r]^{f_{mor}} & G_1 \ar[d]^{(d,c)} \\ H_0 \times H_0 \ar[r]^{(f_{ob},f_{ob})} & G_0\times G_0.}
\]
But this is clear, for if $[\al] \neq [\beta]$, then the preimage of $X^{\im \al} \times X^{\im \beta}$ on the right hand side of the square is empty.

Now it follows from Corollary 4.8 of \cite{Carchedi} that the fat geometric realization of the map $f$ is an equivalence. The nerves of the topological groupoids are ``good'' simplicial spaces in the sense of \cite{Segal-categories} Definition A.4 and this implies that the geometric realization is canonically equivalent to the fat geometric realization by Theorem A.1 of \cite{Segal-categories}. We conclude that the geometric realization of $f$ is an equivalence.
\end{proof}

When multiple groups are in use, we will write $\Twist_{h}^{G}(-)$ to make it clear what group $\Twist_h(-)$ depends on.

A critical property of $\Twist_{h}(-)$ is the way that it interacts with abelian subgroups of $G$.

\begin{prop}
Let $H \subseteq G$ be an abelian subgroup of $G$, then 
\[
\Twist_{h}^{G}(G/H) \simeq \Twist_{h}^{H}(*).
\]
\end{prop}
\begin{proof}
We produce an equivalence of topological groupoids that provides the desired equivalence after applying geometric realization. The map of topological groupoids takes the form
\[
\xymatrix{\Coprod{(\al,\al ') \in \hom(\Z_{p}^{h},H)^{\times 2}}T(\al,\al ')\times *^{\im \al} \ar[d]_{d} \ar@<1ex>[d]^{c} \ar[r] & \Coprod{(\beta,\beta ') \in \hom(\Z_{p}^{h},G)^{\times 2}}T(\beta,\beta ')\times (G/H)^{\im \beta} \ar[d]_{d} \ar@<1ex>[d]^{c} \\ \Coprod{\al \in \hom(\Z_{p}^{h},H)} *^{\im \al} \ar[r] & \Coprod{\beta \in \hom(\Z_{p}^{h},G)} (G/H)^{\im \beta}}
\]

Let $i: H \lra{} G$ denote the inclusion and let the map be defined on objects by sending $* \in *^{\im \al}$ to $eH \in (G/H)^{\im i \circ \al}$ and on maps by sending
\[
T(\al,\al ')\times*^{\im \al}
 \lra{} T(i \circ \al, i \circ \al ') \times eH
\]
via the inclusion $C(\al, \al ') \lra{} C(i \circ \al, i \circ \al ')$.

We prove that the map is fully faithful and essentially surjective. To prove that it is fully faithful we observe that the following square is a pullback:
\[
\xymatrix{\Coprod{(\al,\al ') \in \hom(\Z_{p}^{h},H)^{\times 2}}T(\al,\al ')\times *^{\im \al} \ar[d]_{(d,c)} \ar[r] & \Coprod{(\beta,\beta ') \in \hom(\Z_{p}^{h},G)^{\times 2}}T(\beta,\beta ')\times (G/H)^{\im \beta} \ar[d]_{(d,c)} \\ \Coprod{\al \in \hom(\Z_{p}^{h},H)} *^{\im \al} \times \Coprod{\al \in \hom(\Z_{p}^{h},H)} *^{\im \al} \ar[r] & \Coprod{\beta \in \hom(\Z_{p}^{h},G)} (G/H)^{\im \beta} \times \Coprod{\beta \in \hom(\Z_{p}^{h},G)} (G/H)^{\im \beta}.}
\]
Note that $eH \in (G/H)^{\im \beta}$ if and only if $\im \beta \subseteq H$. Also note that $T(\al, \al ') = \emptyset$ unless $\al = \al '$ because $H$ is abelian. The preimage of a point $(eH,eH) \in (G/H)^{\im \beta} \times (G/H)^{\im \beta '}$ consists of the collection $[g,t] \in T(\beta,\beta ')$ such that 
\[
gH = eH \in (G/H)^{\im g\beta g^{-1} = \beta '}.
\]
To have the equality $gH = eH$, $g$ must be an element of $H$. Thus $\beta = \beta '$ since $H$ is abelian. The square is a pullback because the subspace
\[
\{[g,t]| g \in H\} \subseteq T(\beta,\beta) 
\]
is homeomorphic to $T(\al,\al)$.

To prove that the map of topological groupoids is essentially surjective we must show that $c \circ \pi_1$:
\begin{multline*}
\Big(\Coprod{(\beta,\beta ') \in \hom(\Z_{p}^{h},G)^2}T(\beta,\beta ')\times (G/H)^{\im \beta}\Big) \times_{\big(\Coprod{\beta \in \hom(\Z_{p}^{h},G)} (G/H)^{\im \beta}\big)} \Big(\Coprod{\al \in \hom(\Z_{p}^{h},H)} *^{\im \al}\Big) \\
\downarrow \\
\Coprod{\beta \in \hom(\Z_{p}^{h},G)} (G/H)^{\im \beta} \\
\end{multline*}
has local sections. As the codomain is a set it suffices to show that the map is surjective. Given
\[
gH \in (G/H)^{\im \beta '},
\]
we will produce an element in the domain that maps to it. Let $\beta = g^{-1}\beta 'g$. Then 
\[
eH = g^{-1}gH \in (G/H)^{\im \beta}
\]
so $\im \beta \subseteq H$. Now $[g,t] \in T(\beta,\beta ')$ sends $eH \in (G/H)^{\im \beta}$ to $gH \in (G/H)^{\im \beta '}$. Thus the map is a surjection.
\end{proof}

We note some more properties of $T(-)$. 
\begin{prop}
Let $\text{Group}_{\Z_{p}^h/}^{fin}$ be the full subcategory of the category of topological groups and continuous maps under $\Z_{p}^h$ consisting of the finite groups. The map
\[
T: \al \mapsto T(\al)
\]
extends to a functor
\[
T: \text{Group}_{\Z_{p}^h/}^{fin} \lra{} \text{Group}.
\]
\end{prop}
\begin{proof}
This follows in a straight-forward way from the fact that 
\[
C: \text{Group}_{\Z_{p}^h/}^{fin} \lra{} \text{Group},
\]
which sends $\al \mapsto C(\im \al)$, is a functor.
\end{proof}

\begin{prop}
\label{prod}
Let $\al:\Z_{p}^{h} \lra{} G$, $\beta:\Z_{p}^{h} \lra{} H$ and $\al\times\beta: \Z_{p}^{h} \lra{} G\times H$ be the product, then 
\[
T(\al \times \beta) \cong T(\al) \times_{\QZ^{h}} T(\beta).
\]
\end{prop}
\begin{proof}
First note that the pullback does make sense because $T(\al)$ and $T(\beta)$ both come with canonical maps to $\QZ^{h}$ that send
\[
[g,t] \mapsto t.
\]
The isomorphism is given by
\[
[(g,h),t] \mapsto ([g,t],[h,t]).
\]
\end{proof}

Finally we give the construction of the topological part of the twisted character map from height $n$ to height $t$. Fix a finite group $G$. We produce it as the realization of a map of topological groupoids.

Let $k$ be such that every map $\Z_{p}^{n-t} \lra{} G$ factors through $\Lk = (\Zp{k})^{n-t}$. Let $\gamma_k:\Z_{p}^{n-t} \lra{} \Lk$ be the canonical quotient.

We begin by extending the morphism space of $\Twist_{n-t}^G(X)$.
\begin{prop}
The morphism space of $\Twist_{n-t}^G(X)$ can be extended to 
\[
\Coprod{[\al] \in \hom(\Z_{p}^{n-t},G)/\sim} T(\gamma_k\times \al)\times X^{\im \al}.
\]
\end{prop}
\begin{proof}
Let $l \in \Lk$. We define the action to be
\[
[l,g,t]x = \al(l)gx = gx \in X^{\im \al}.
\]
\end{proof}

By Prop \ref{prod} we have that
\begin{align*}
\Coprod{[\al] \in \hom(\Z_{p}^{n-t},G)/\sim} T(\gamma_k\times \al)\times X^{\im \al}
&\cong \Coprod{[\al] \in \hom(\Z_{p}^{n-t},G)/\sim} T(\gamma_k)\times_{\QZ^{n-t}} T(\al)\times X^{\im \al} \\
&\cong T(\gamma_k) \times_{\QZ^{n-t}} \Coprod{[\al] \in \hom(\Z_{p}^{n-t},G)/\sim}T(\al)\times X^{\im \al}.
\end{align*}
It is easiest to think about this final term as being the morphisms of a topological groupoid over the object space
\[
*\times_* \Coprod{[\al] \in \hom(\Z_{p}^{n-t},G)/\sim} X^{\im \al}
\]
by collapsing $T(\gamma_k)$ and $\QZ^{n-t}$ to a point. This implies that the classifying space of the topological groupoid
\[
\xymatrix{\Coprod{[\al] \in \hom(\Z_{p}^{n-t},G)/\sim} T(\gamma_k\times \al)\times X^{\im \al} \ar[d]_{d} \ar@<1ex>[d]^{c} \\ \Coprod{[\al] \in \hom(\Z_{p}^{n-t},G)/\sim} X^{\im \al}}
\]
is
\[
BT(\gamma_k) \times_{B\QZ^{n-t}} \Coprod{[\al] \in \hom(\Z_{p}^{n-t},G)/\sim}ET(\al)\times_{T(\al)} X^{\im \al}
\]
or
\[
BT(\gamma_k) \times_{B\QZ^{n-t}} \Twist_{n-t}^G(X).
\]
The pullback is in fact a homotopy pullback because $T(\gamma_k) \lra{} \QZ^{n-t}$ is surjective and so gives rise to a fibration when $B(-)$ is applied.

The topological part of the character map
\[
BT(\gamma_k) \times_{B\QZ^{n-t}} \Twist_{n-t}^G(X) \lra{} EG\times_G X
\]
is constructed by realizing a map of topological groupoids
\[
\xymatrix{\Coprod{[\al] \in \hom(\Z_{p}^{n-t},G)/\sim} T(\gamma_k\times \al)\times X^{\im \al} \ar[d]_{d} \ar@<1ex>[d]^{c} \ar[r]^(.75){f_{mor}} & G \times X \ar[d]_{d} \ar@<1ex>[d]^{c} \\ \Coprod{[\al] \in \hom(\Z_{p}^{n-t},G)/\sim} X^{\im \al} \ar[r]^(.65){f_{ob}} & X.
 }
\]
We will define the map on objects $f_{ob}$ and the map on morphisms $f_{mor}$. The map $f_{ob}$ is defined as the coproduct of the inclusions
\[
X^{\im \al} \hookrightarrow X.
\]
The map $f_{mor}$ is a bit more complicated. For 
\[
([l,g,t],x) \in T(\gamma_k\times \al)\times X^{\im \al},
\]
we define
\[
f_{mor}:([l,g,t],x) \mapsto (g\al(l)^{-1},x).
\]
Note that the order switches and $\al(l)$ is inverted. We show this is well-defined. For $i \in \Z_{p}^{n-t}$,
\begin{eqnarray*}
(g\al(l)^{-1},x) = f_{mor}([l,g,t+i],x) = f_{mor}([l+\gamma_k(i),g\al(i),t],x) \\ = ([g\al(i)\al(\gamma_k(i))^{-1}\al(l)^{-1}],x) = ([g\al(l)^{-1}],x), 
\end{eqnarray*}
using the fact that $\al(\gamma_k(i)) = \al(i)$.
It is clear that the square made up of $f_{ob}$, $f_{mor}$, and the domains (projections) commutes. We show that the diagram involving the codomain maps commutes:
\[
\xymatrix{([l,g,t],x) \ar[r] \ar[d] & (g\al(l)^{-1},x) \ar[d] \\ \al(l)gx = gx \ar[r] & g\al(l)^{-1}x = gx.}
\]

As an example of the above construction we compute the topological map from $X = *$ and $G = \Zp{k}$.
\begin{example} \label{salpha}
We compute
\[
BT(\gamma_k)\times_{B
\QZ^{n-t}} \Twist_{n-t}^{\Zp{k}}(*) \lra{} B\Zp{k}.
\]
Because $\Zp{k}$ is abelian each map $\al:\Z_{p}^{n-t} \lra{} \Zp{k}$ is its own conjugacy class. Thus on objects we get
\[
\Coprod{\al \in \hom(\Z_{p}^{n-t},\Zp{k})} * \lra{} *.
\]
Fix an $\al$ as above. On morphisms in the path component corresponding to $\al$ we have the subtraction group homomorphism
\[
s_{\al}:T(\gamma_k \times \al) = (\Lk \oplus \Zp{k}) \oplus_{\Z_{p}^{n-t}}\Q_{p}^{n-t}  \lra{} \Zp{k},
\] 
which sends
\[
s_{\al}: (l,g,t) \mapsto g - \al(l).
\]
When $\al$ is the zero map this is just the projection on to $\Zp{k}$. Put all together on morphisms we get the disjoint union over the maps $\al$ of the subtractions maps
\[
\Coprod{\al \in \hom(\Z_{p}^{n-t},\Zp{k})}s_{\al}.
\]
The realization of this map is thus
\[
\Coprod{\al \in \hom(\Z_{p}^{n-t},\Zp{k})}Bs_{\al}: \Coprod{\al \in \hom(\Z_{p}^{n-t},\Zp{k})}BT(\gamma_k \times \al) \lra{} B\Zp{k},
\]
or written another way a map
\[
BT(\gamma_k)\times_{B\QZ^{n-t}}\Coprod{\al \in \hom(\Z_{p}^{n-t},\Zp{k})}BT(\al) \lra{} B\Zp{k}.
\]
\end{example}

\subsection{The Algebraic Map}
The algebraic part of the twisted character map uses the properties of the ring $B_t$ constructed and described in Section \ref{Bt} to construct an appropriate codomain for the twisted character map.

The discussion regarding gradings in \cite{tgcm} carries over to this situation.

Applying $\E$ to the topological map constructed in the last section we get 
\[
\E^0(EG\times_G X) \lra{} \E^0(BT(\gamma_k) \times_{B\QZ^{n-t}} \Twist_{n-t}^G(X)). 
\]
We begin with an algebraic manipulation of the codomain of the topological part of the character map.
\begin{prop} \label{kunneth}
There is an isomorphism
\[
\E^0(BT(\gamma_k) \times_{B\QZ^{n-t}} \Twist_{n-t}^G(X)) \cong \E^0(BT(\gamma_k)) \otimes_{\E^0(B\QZ^{n-t})} \E^0(\Twist_{n-t}^G(X)).
\]
\end{prop}
\begin{proof}
Let us assume that $n>0$ and let $\mathbb{T} = (S^{1})^{\times n-t}$. The result follows from the fact that the functors on $\mathbb{T}$-spaces (with no finiteness hypotheses)
\[
\E^0(BT(\gamma_k) \times_{B\mathbb{T}} E\mathbb{T}\times_{\mathbb{T}} Y)
\]
and
\[
\E^0(BT(\gamma_k)) \otimes_{\E^0(B\mathbb{T})} \E^0(E\mathbb{T}\times_{\mathbb{T}} Y)
\]
are both cohomology theories. This includes the spaces in the proposition by Remark \ref{T-action}. The first functor is a cohomology theory in $Y$ because it is the pullback along a fibration and this pullback commutes with all homotopy colimits by \cite{hocolims}. The second is a cohomology theory in $Y$ because it is extension by a finitely generated free module. There is a natural map from the tensor product to the other functor. It suffices to check that these cohomology theories agree on $\mathbb{T}/A$ for any closed subgroup $A \subseteq \mathbb{T}$. Now as $\mathbb{T}/A \cong S^1/A_1 \times \ldots \times S^1/A_{n-t}$ it suffices to check these one at a time and here the isomorphism is clear. 
\end{proof}

\begin{example} \label{Es}
Let $s_\al$ be as in Example \ref{salpha} in the previous section. We compute $\E^0(Bs_{\al})$. We begin by computing $\E^0(BT(\al))$. For a complete $\E^0$-algebra $R$ we have the correspondence
\[
\hom_{\text{cont } \E^0}(\E^0(B\QZ^{n-t}),R) \cong \hom_{ab}(\Z_{p}^{n-t},\G_{\E}(R)),
\]
where $\Z_{p}^{n-t} \cong (\QZ^{n-t})^*$. This carries over to $T(\al)$.
\[
\hom_{\text{cont } \E^0}(\E^0(BT(\al)),R) \cong \hom_{ab}(T(\al)^*,\G_{\E}(R)).
\]
A presentation for $T(\al)^*$ as a $\Z_p$-module was given in Example \ref{presentation} and this gives a description of $\E^0(BT(\al))$ as an $\E^0(B\QZ^{n-t})$-algebra. Using the coordinate and the presentation we have
\[
\E^0(BT(\al)) \cong \E^0\powser{q_1,\ldots,q_{n-t},x}/([p^k](x) - ([i_1](q_1) +_{\G_{\E}} \ldots +_{\G_{\E}} [i_{n-t}](q_{n-t}))).
\]
Recall that 
\[
\E^0(BT(\gamma_k)) \cong \E^0\powser{\bq,\bA}/([p^k]\bA - \bq).
\]
The map $s_{\al}$ factors through $T(\al \times \al)$:
\[
\xymatrix{T(\gamma_k\times\al) \ar[r]^{s_{\al}} \ar[d] & \Zp{k} \\ T(\al \times \al) \ar[ru]^s &}.
\]
Now dualizing $s$ gives the map (in the notation of Example \ref{presentation})
\[
(\Zp{k})^* \lra{} T(\al \times \al)^* = \{b_1,\ldots,b_{n-t},x,y|p^kx = p^ky = i_1b_1 + \ldots i_{n-t} b_{n-t}\},
\]
defined by
\[
1 \mapsto y-x.
\]
We get the following explicit description of $\E^0(Bs_\al)$
\[
\E^0(Bs_{\al}): x \mapsto (x -_{\G_{\E}} ([i_1]A_1+_{\G_{\E}} \ldots +_{\G_{\E}} [i_{n-t}]A_{n-t})).
\]
\end{example}

The algebraic part of the twisted character map will be assembled from two maps: the canonical map of spectra 
\[
L:\E \lra{} \LE
\]
and the canonical flat map of rings
\[
i_{t}:\E^*(BT(\gamma_k)) \lra{} B_{t}^{*}.
\]

The algebraic part of the twisted character map is
\begin{align*}
\E^*(BT(\gamma_k)) \otimes_{\E^*(B\QZ^{n-t})} \E^*(\Twist_{n-t}^G(X)) &  \\
\lra{}  B_{t}^{*}\otimes_{\LE^*(B\QZ^{n-t})} & \LE^*(\Twist_{n-t}^G(X)).
\end{align*}
It is the tensor product of the maps $i_t$ and $L(\Twist_{n-t}^{G}(X))$ over $L(B\QZ^{n-t})$. 

This map composed with the topological map of the previous section gives the twisted character map
\[
\Upsilon_G:\E^*(EG \times_G X) \lra{} B_{t}^{*}\otimes_{\LE^*(B\QZ^{n-t})} \LE^*(\Twist_{n-t}^{G}(X)).
\]

When we want to specify a space $X$ in the twisted character map we will write $\Upsilon_G(X)$ for the map above. We will use a shorthand for the codomain of the twisted character map. Let
\[
B_{t}^*(\Twist_{n-t}^{G}(X)) = B_{t}^{*}\otimes_{\LE^*(B\QZ^{n-t})} \LE^*(\Twist_{n-t}^{G}(X)).
\]

\begin{prop}
The map $\Upsilon_G$ is independent of the choice of $k$.
\end{prop}
\begin{proof}
This follows from the proof of Proposition 3.13 in \cite{tgcm}.
\end{proof}

\begin{thm}
The map on global sections of the $p^k$-torsion of \pdiv groups
\[
\xymatrix{\G_0\oplus_{\Z_{p}^{n-t}}\Q_{p}^{n-t}[p^k] \ar[r] \ar[d] & \G_{\E}[p^k] \ar[d] \\
\Spf_{I_t + (\bq)}(B_t) \ar[r] & \Spec(\E^0)}
\]
is recovered by $\Upsilon_{\Zp{k}}(*)$.
\end{thm}
\begin{proof}
We will use the results of Section \ref{Bt} and the examples of the previous sections. By Example \ref{salpha} the topological part of the map gives
\[
\Prod{\al \in \hom(\Z_{p}^{n-t},\Zp{k})}\E^0(Bs_{\al}): \E^0(B\Zp{k}) \lra{} \Prod{\al \in \hom(\Z_{p}^{n-t},\Zp{k})} \E^0(BT(\gamma_k \times \al)).
\]
By Example \ref{Es} this map sends
\[
x \mapsto \Prod{i\in \Lk} (x -_{\G_{\E}} [i](\bA)).
\]
Now composing with the map to $B_t$ on the left and the map to $\LE^0(BT(\al))$ on the right sends $\bA \in \E^0(BT(\gamma_k))$ to $\bA \in B_t$ and $x$ and $\bq \in \E^0(BT(\al))$ to 
$x$ and $\bq \in \LE^0(BT(\al))$. All together this gives the map described in Section \ref{Bt}: the global sections of the $p^k$-torsion of the map of \pdiv groups above.
\end{proof}

\subsection{The Isomorphism}
In this section we prove by reduction to the case of $G$ finite abelian and $X = *$ that for any finite $G$ and finite $G$-space $X$ there is an isomorphism
\[
B_t \otimes_{\E^0} \Upsilon_G:B_t \otimes_{\E^0} \E^*(EG \times_G X) \lra{\cong} B_{t}^{*}\otimes_{\LE^*(B\QZ^{n-t})} \LE^*(\Twist^{G}_{n-t}(X)).
\]

The proof here follows the same lines as that in \cite{tgcm}. Because of this we will only point out how to prove the essential properties of the $\Twist_{n-t}(-)$ functor that allow for the reduction to cyclic $p$-groups where it differs from the proof in \cite{tgcm}.

\begin{prop}
The functor $\Twist_{n-t}(-)$ commutes with pushouts of finite $G$-CW complexes.
\end{prop}
\begin{proof}
This is similar to the proof in \cite{tgcm}. The main difference is that the Borel construction in the definition of $\Twist_{n-t}(-)$ has an infinite group. The main point of the proof is that everything in sight is a colimit and the groups of the form $T(\al)$ for $\al:\Z_{p}^{n-t} \lra{} G$ are no exception. 
Let $T(\al)_i$ be the pushout in the sense of Proposition \ref{pushout}
\[
\xymatrix{\Z_{p}^{n-t} \ar[r]^{\times p^i} \ar[d]_{\al} & \Z_{p}^{n-t} \ar[d] \\ C(\im \al) \ar[r] & T(\al)_i. }
\]
The group $T(\al)_i$ is finite and
\[
T(\al) \cong \Colim{i} \text{ }T(\al)_i.
\]
\end{proof}

\begin{prop}
The functor $\Twist_{n-t}(-)$ commutes with geometric realization of simplicial finite $G$-CW complexes.
\end{prop}
\begin{proof}
This follows from \cite{tgcm}. There are no difficulties in extending the result there for the functor $\Fix_{n-t}(-)$ to $\Twist_{n-t}(-)$.
\end{proof}

Now we follow \cite{tgcm}. Using the Bousfield-Kan spectral sequence the two facts above allow us to reduce the isomorphism for transchromatic generalized character maps to the case of finite $G$-CW complexes with abelian stabilizers. Now Mayer-Vietoris reduces this to the case of $BA$ for $A$ a finite abelian group. It may not be entirely clear that the cohomology theory in the codomain of the twisted character map above has the Kunneth isomorphism that we need to reduce to cyclic p-groups. We prove that now.
\begin{prop}
Let $G \times H$ be a finite abelian group. Then
\begin{align*}
B_{t}^{*}(\Twist_{n-t}^{G\times H}(*)) &\cong \\
& B_{t}^{*}(\Twist_{n-t}^G(*)) \otimes_{B_{t}^{*}(\Twist_{n-t}^{e}(*))} B_{t}^{*}(\Twist_{n-t}^{H}(*)).
\end{align*}
\end{prop}
\begin{proof}
This follows from the chain of homotopy equivalences
\begin{align*}
\Twist_{n-t}^{G \times H}(*) &\simeq \Coprod{\al \in \hom(\Z_{p}^{n-t},G\times H)} BT(\al) \\
&\simeq \Coprod{\al \in \hom(\Z_{p}^{n-t},G\times H)} B(T(\al_G) \times_{\QZ^{n-t}} T(\al_H)) \\
&\simeq \Coprod{\al_G \in \hom(\Z_{p}^{n-t},G)} BT(\al_G) \times_{B\QZ^{n-t}} \Coprod{\al_H \in \hom(\Z_{p}^{n-t},H)} BT(\al_H) \\
&\simeq \Twist_{n-t}^G(*) \times_{\Twist_{n-t}^e(*)} \Twist_{n-t}^H(*).
\end{align*}
The result now follows from the fact that the map
\[
\Twist_{n-t}^G(*) \lra{} \Twist_{n-t}^e(*)
\]
is a fibration and the cohomology of the domain is a finitely generated free module over the cohomology of the codomain just as in the proof of Proposition \ref{kunneth}.
\end{proof}

The above propositions together with the work in \cite{tgcm} establish the main theorem.
\begin{thm}
The transchromatic twisted character map $\Upsilon_{G}$ induces an isomorphism when tensored up to $B_t$
\[
B_t \otimes_{\E^0} \Upsilon_G:B_t \otimes_{\E^0} \E^*(EG \times_G X) \lra{\cong} B_{t}^{*}(\Twist_{n-t}^{G}(X)).
\]
\end{thm}

\subsection{Relation to \cite{tgcm}}
There is a canonical topologically-induced quotient map $B_t \lra{} \Chat$ and a canonical natural transformation 
\[
EG \times_G \Fix_{n-t}(-) \lra{} \Twist_{n-t}(-). 
\]
This can be used to recover the character map of \cite{tgcm}.

The quotient map was described in Section \ref{relationtoC_t}. Recall that
\[
EG\times_G \Fix_{n-t}(X) \simeq \Coprod{[\al] \in \hom(\Z_{p}^{n-t},G)/\sim} EC(\im \al) \times_{C(\im \al)} X^{\im \al}.
\]
The natural transformation is induced by the inclusion 
\[
C(\im \al) \hookrightarrow T(\al).
\]
We can map $EG\times_G \Fix_{n-t}(X)$ to $\Twist_{n-t}(X)$ on components via the inclusion above. Putting the map of rings and the map of spaces together we get the map of equivariant cohomology theories
\begin{align*}
B_{t}\otimes_{\LE^0(B\QZ^{n-t})} \LE^*(&\Twist_{n-t}(X)) \\ \lra{} & \Chat \otimes_{\LE^0} \LE^*(EG\times_G \Fix_{n-t}(X)).
\end{align*}
Note that $\G_0 \oplus \QZ^{n-t}$ is the pullback: 
\[
\xymatrix{\G_0 \oplus \QZ^{n-t} \ar[r] \ar[d] & \G_0\oplus_{\Z_{p}^{n-t}}\Q_{p}^{n-t} \ar[d] \\ \Spf_{I_t}(\Chat) \ar[r] & \Spf_{I_t + (\bq)}(B_t).}
\]
This implies that the composite
\[
\E^*(EG \times_G X) \lra{} B_{t}^{*}(\Twist_{n-t}(X)) \lra{} \Chat^{*}(EG\times_G \Fix_{n-t}(X))
\]
recovers a completed version of the character map $\Phi_G$ of \cite{tgcm}.
\bibliographystyle{abbrv}
\bibliography{mybib}

\end{document}